\documentclass[a4paper,
               headinclude=false,
               headsepline,
               footinclude=false,
               footsepline,
               plainfootsepline,
               abstract=true,
               twoside]{scrartcl}
\usepackage[active]{srcltx} 
\usepackage[T1]{fontenc}
\usepackage{scrlayer-scrpage}
\usepackage{amssymb,amsmath,amsthm}
\usepackage{hyperref}
\usepackage{pxfonts}
\newcommand{\UU}{\ensuremath{\mathbb{U}}}
\newcommand{\VV}{\ensuremath{\mathbb{V}}}
\newcommand{\cB}{\ensuremath{\mathcal{B}}}
\newcommand{\cL}{\ensuremath{\mathcal{L}}}
\newcommand{\AutUL}{\Gamma}
\newcommand{\tg}[1]{\AutUL_{[#1]}}
\newcommand{\set}[2]{\left\{{#1}\left|\vphantom{#1#2\strut}\right.\, 
                    {#2\strut}\right\}}
\newcommand{\smallset}[2]{\{{#1}\left|\vphantom{}\right.\,{#2}\}}
\newcommand{\genset}[2]{\langle{#1}\,\vert\,{#2}\rangle}
\newcommand{\Aut}[2][]{\operatorname{Aut}_{#1}(#2)}
\newcommand{\id}{\ensuremath{\mathrm{id}}}
\newcommand{\Alt}[1]{\operatorname{A}_{#1}}
\newcommand{\Sym}[1]{\operatorname{S}_{#1}}
\newcommand{\SL}[2]{\operatorname{SL}(#1,#2)}
\newcommand{\PSL}[2]{\operatorname{PSL}(#1,#2)}
\newcommand{\GL}[2]{\operatorname{GL}(#1,#2)}
\newcommand{\PgL}[2]{\operatorname{P\Gamma L}(#1,#2)}
\newcommand{\AGL}[2]{\operatorname{AGL}(#1,#2)}
\newcommand{\ASL}[2]{\operatorname{ASL}(#1,#2)}
\newcommand{\U}[2]{\operatorname{U}(#1,#2)}
\newcommand{\SU}[2]{\operatorname{SU}(#1,#2)}
\newcommand{\PSU}[2]{\operatorname{PSU}(#1,#2)}
\newcommand{\Sp}[2]{\def\test{#1}\ifx\empty\test
      \mathrm{Sp}\left(#2\right)\else\mathrm{Sp}(#1,#2)\fi}
\newcommand{\Sz}[1]{\mathrm{Sz}(#1)}
\newcommand{\Ree}[1]{\mathrm{Ree}(#1)}
\newcommand{\Cg}[1]{\operatorname{C}\nolimits_{#1}}
\newcommand{\Q}[1]{\operatorname{Q}\nolimits_{#1}}
\newcommand{\FF}{\mathbb{F}}
\newcommand{\gp}{G}
\newcommand{\lpg}{G^\dagger}
\makeatletter
\newcommand{\widebar}[1]{\mathop {\mathchoice
    {\vbox
    {\m@th \ialign {##\crcr \noalign {\kern 1\p@ }\hrulefill \crcr
        \noalign {\kern 1\p@ \nointerlineskip }%
        $\hfil \textstyle {#1}\hfil $\crcr }}}
    {\vbox
    {\m@th \ialign {##\crcr \noalign {\kern 1\p@ }\hrulefill \crcr
        \noalign {\kern 1\p@ \nointerlineskip }%
        $\hfil \textstyle {#1}\hfil $\crcr }}}
    {\vbox
    {\m@th \ialign {##\crcr \noalign {\kern 1\p@ }\hrulefill \crcr
        \noalign {\kern 1\p@ \nointerlineskip }%
        $\hfil \scriptstyle {#1}\hfil $\crcr }}}
    {\vbox
    {\m@th \ialign {##\crcr \noalign {\kern 1\p@ }\hrulefill \crcr
        \noalign {\kern 1\p@ \nointerlineskip }%
        $\hfil \scriptscriptstyle {#1}\hfil $\crcr }}}%
    }}
\makeatother
\let\gal\widebar
\addtokomafont{title}{\rmfamily}
\addtokomafont{section}{\rmfamily}
\addtokomafont{subsection}{\rmfamily}
\swapnumbers
\newtheorem*{main}{Main Theorem}
\theoremstyle{plain}
\newtheorem{theo}{Theorem}[section]
\newtheorem{coro}[theo]{Corollary}
\newtheorem{lemm}[theo]{Lemma}
\newtheorem{prop}[theo]{Proposition}

\theoremstyle{definition}

\newtheorem{rema}[theo]{Remark}

\let\setminus\smallsetminus
\let\epsilon\varepsilon
\let\phi\varphi
\lohead[]{\footnotesize\normalfont\rmfamily \MYtitle}
\lehead[]{\footnotesize\normalfont\rmfamily \MYauthor}
\rehead[]{\footnotesize\normalfont\rmfamily \MYtitle}
\rohead[]{\footnotesize\normalfont\rmfamily \MYauthor}
\lefoot[\pagemark]{\pagemark}
\rofoot[\pagemark]{\pagemark}
\title{Moufang sets generated by translations\\ in unitals}%
\author{Theo Grundh\"ofer, Markus J. Stroppel, Hendrik Van Maldeghem}%
\let\MYauthor\shortauthor
\let\MYtitle\shorttitle
\newcommand{\keywords}[1]{\par\noindent{\bfseries Keywords: }#1}
\newcommand{\classification}[1]{\par%
  \noindent{\bfseries Mathematics Subject Classification: }#1}
\begin{document}
\maketitle

\begin{abstract}
  \noindent%
  We consider unitals of order~$q$ with two points which are centers
  of translation groups of order~$q$. The group $G$ generated by these
  translations induces a Moufang set on the block joining the two
  points. We show that $G$ is either
  $\operatorname{SL}(2,\mathbb{F}_q)$ (as in all classical unitals and
  also in some non-classical examples), or
  $\operatorname{PSL}(2,\mathbb{F}_q)$, or a Suzuki or a Ree group.
  Moreover, $G$ is semi-regular outside the special
  block.  %

  \medskip
  \classification{%
    05E20 
    05B30 
    51A10 
  }%
\end{abstract}

\noindent%
In~\cite{MR3090721} we considered unitals admitting all possible
translations (see Section~\ref{sec:unitals} below for definitions) and
characterized the classical (hermitian) unitals by this property. %
The present paper takes a more general view: we only assume
translations with centers on a single block, and prove the following.

\begin{main}
  Let\/ $\UU$ be a unital of order~$q$ with two points which are
  centers of translation groups of order~$q$.  Then the group~$G$
  generated by these two translation groups is isomorphic to one of
  the following.
  \begin{enumerate}
  \item $\SL 2 {\FF_q} $ or $\PSL 2 {\FF_q}$, where $q$ is a prime
    power.
  \item 
    The Suzuki group $\Sz{q}$, where $q = 2^{2s} \ge 2^6$ for some odd
    integer $s\ge 3$.
  \item 
    The Ree group $\Ree{q}$, where $q = 3^{3r} \ge 3^3$ for some
    odd integer $r\ge 1$.
  \end{enumerate}
  Moreover, the group~$G$ acts semi-regularly on the set of points
  outside the block containing the translation centers.
\end{main}

In the classical (i.e.\ hermitian) unital of order $q$, the group $G$
as above is isomorphic to $\SL 2 {\FF_q}$, %
cp.~\cite[3.1, 4.1]{MR3533345}.  It seems that no unital of odd order
$q$ is known where $G \cong \PSL2{\FF_q}$; %
for $q=3$ there is no such unital by Proposition \ref{PSL23} below
(this uses \cite[2.3]{MR3090721}).  There exists a non-classical
unital of order $q=4$ such that $G \cong \SL 2 {\FF_4} \cong \Alt 5$,
see \cite[4.1]{MR3533345}.  More examples (of order $8$) have been
found by M\"ohler in her Ph.D. thesis~\cite[Section\,6]{Moehler2020}.

The unitals of order~$q$ with two points which are centers of
translation groups of order~$q$ are studied also by Rizzo in his
Ph.D. thesis~\cite{RizzoPhD}. The last chapter of this thesis contains
results about embeddings of such unitals into projective planes of
order~$q^2$.
For generalized quadrangles, a situation analogous to the one
considered in the present paper is treated in a series of papers,
see~\cite{MR1895347}, \cite{MR1957745}, \cite{MR1895346},
and~\cite{MR2245777}.

Our Main Theorem resembles results for projective planes (instead of unitals)
obtained by Hering \cite{MR0467511}, \cite{MR533094}, who considered groups
generated by elations. The following statement is a very special case
of \cite[Theorem 3.1]{MR533094}: if a projective plane of finite order $q$
contains a triangle $p, z_1, z_2$ such that the group of all elations
with center $z_j$ and axis $pz_j$ has order $q$ for $j=1,2$, then the plane is
desarguesian and the group generated by these two elation groups
is isomorphic to $\SL 2 {\FF_q}$.

\section{Unitals, translations and Moufang sets}
\label{sec:unitals}

A unital $\UU = (U,\cB)$ of order~$q>1$ is a
$2$-$(q^3+1,q+1,1)$-design.  In other words, $\UU$ is an incidence
structure such that any two points in~$U$ are joined by a unique block
in~$\cB$, there are $|U|=q^3+1$ points, and every block has exactly
$q+1$ points. It follows that every point is on exactly $q^2$ blocks.

\begin{lemm}\label{LemFix}
  Let\/ $\VV$ be a unital of order~$q$, and let\/ $\varphi\in\Aut\VV$
  be an automorphism of\/~$\VV$. If\/~$\varphi$ fixes more than
  $q^2+q-2$ points then~$\varphi$ is trivial. %
  In particular, if~$\varphi$ fixes a point~$x$ and a block~$B$ not
  through~$x$ and all points on blocks joining~$x$ to a point on~$B$
  then~$\varphi$ is trivial. %
\end{lemm}
\begin{proof}
  Let $y$ be a point that is moved by~$\varphi$. Joining~$y$ with each
  one of the fixed points yields a set of blocks through~$y$. At most
  one of those blocks can be a fixed block of~$\varphi$, and a non-fixed
  block contains at most one fixed point. %
  If a fixed block through~$y$ exists then that block contains at most
  $q-1$ fixed points. For the number~$f$ of fixed points we obtain
  $q^2-1 \ge f - (q-1)$ and $f \le q^2+q-2$. If no block through~$y$ is
  fixed then $f \le q^2 \le q^2+q-2$.

  The second assertion follows from the fact that the point set in
  question contains $(q+1)q+1 = q^2+q+1$ points. 
\end{proof}

An automorphism of $\UU$ is called a \emph{translation} of
$\UU$ with center~$z$ if it fixes each block through the
point~$z$. The set of all translations with center~$z$ is denoted
by~$\tg{z}$. %

A \emph{Moufang set} is a set~$X$ together with a collection of groups
$(R_x)_{x\in X}$ of permutations of~$X$ such that each~$R_x$ fixes~$x$
and acts regularly (i.\,e., sharply transitively) on
$X \setminus \{x\}$, and such that the collection
$\smallset{R_y}{y \in X}$ is invariant under conjugation by the
\emph{little projective group} $\genset{R_x}{x \in X}$ of
the Moufang set. The groups~$R_x$ are called \emph{root groups}.

The finite Moufang sets are known explicitly:

\begin{theo}\label{finiteMoufang}
  The little projective group of a finite Moufang set is either
  sharply two-transitive, or it is permutation isomorphic to one of
  the following two-transitive permutation groups of degree $q+1$:
  $\PSL 2{\FF_q}$ with a prime power $q>3$, $\PSU3{\FF_{f^2}|\FF_f}$ with
  a prime power $q=f^3 \ge 3^3$, a Suzuki group
  $\Sz{2^s}= {}^2B_2(2^s)$ with $q=2^{2s} \ge 2^6$, or a Ree group
  $\Ree{3^r} = {}^2G_2(3^r)$ with $q= 3^{3r}$, where $r$ and $s$ are
  positive odd integers.
\end{theo}

This was proved (in the context of split BN-pairs of rank one) by
Suzuki~\cite{MR0162840} and Shult~\cite{MR0296150} for even~$q$, and
by Hering, Kantor and Seitz~\cite{MR0301085} for odd~$q$; these papers
rely on deep results on finite groups, but not on the classification
of all finite simple groups. See also Peterfalvi~\cite{MR1063141}. %
Note that $\PSL2{\FF_2} \cong\AGL1{\FF_3}$,
$\PSL2{\FF_3} \cong \Alt4 \cong \AGL1{\FF_4}$,
$\PSU3{\FF_4|\FF_2} \cong\ASL2{\FF_3}$ and $\Sz{2}\cong \AGL1{\FF_5}$ are
sharply two-transitive. The smallest Ree group $\Ree{3} \cong \PgL2{\FF_8}$
is almost simple, but not simple.

\bigbreak Let $\UU = (U,\cB)$ be a unital of order~$q$, and let
$\AutUL = \Aut\UU$ be its automorphism group. %
Throughout this paper, we assume that $\UU$ contains two
points~$\infty$ and~$o$ such that for $z\in\{\infty,o\}$ the
translation group~$\tg{z}$ has order~$q$.  Then~$\tg{z}$ acts
transitively on $B\setminus\{z\}$, for any block~$B$ through~$z$
(see~\cite[1.3]{MR3090721}). In particular, $\tg{x}$ has that
transitivity property for each point~$x$ on the block~$B_\infty$
joining~$\infty$ and~$o$. The group~$\gp$ generated by
$\tg{\infty}\cup\tg{o}$ contains the translation group~$\tg{x}$ for
each $x\in B_\infty$, and
$(B_\infty,(\tg{x}|_{B_\infty})_{x\in B_\infty})$ is a Moufang set,
with little projective group
$\lpg \coloneqq \gp|_{B_\infty} \cong \gp/\gp_{[B_\infty]}$, where
$\gp_{[B_\infty]}$ is the kernel of the action on~$B_\infty$.  This
kernel coincides with the center $Z$ of~$\gp$,
see~\cite[3.1.2]{MR3090721} or \cite[2.11]{MR0467511}. So~$\gp$ is a
central extension of the little projective group~$\lpg$.

\begin{coro}\label{applyRigidity}
  The kernel $\gp_{[B_\infty]} = Z$ acts semi-regularly on
  $U\setminus B_\infty$. %
\end{coro}

\begin{proof}
  If $\phi \in \gp_{[B_\infty]} = Z$ fixes $x\in U\setminus B_\infty$,
  then it fixes also $x^g$ for every $g\in G$, hence
   all points on blocks joining~$x$ to a point on~$B_\infty$.
  Thus Lemma \ref{LemFix} implies that $\phi$ is trivial.
\end{proof}

The following fact was observed in the proof of \cite[Theorem 2.4]{MR533094}.

\begin{lemm}\label{simpleMoufangSetCommutators}
  Let\/ $(X,\smallset{\Delta_x}{x\in X})$ be a finite Moufang set.
  If the little
  projective group~$\Phi=\langle \Delta_x \mid x\in X\rangle$ is simple then
  $\Delta_x = [\Delta_x, \Phi_x]$ for every $x\in X$, where $\Phi_x$
  denotes the stabilizer of $x$ in $\Phi$. 
\end{lemm}

\begin{proof}
  By the classification of finite Moufang sets, see
  \ref{finiteMoufang}, the simple group~$\Phi$ is isomorphic to
  $\PSL2{\FF_q}$, $\PSU3{\FF_{t^2}|\FF_t}$, $\Sz{2^s}$,
  or~$\Ree{3^r}$, where $q>3$ with $q+1=|X|$, $t>2$ with $t^3+1=|X|$,
  $s>1$ with $2^{2s}+1=|X|$, or $r>1$ with $3^{3r}+1=|X|$,
  respectively. %
  We have $[\Delta_x, \Phi_x] \le \Delta_x$ since $\Phi_x$ normalizes
  $\Delta_x$; it remains to show that
  $\Delta_x \le [\Delta_x, \Phi_x]$.  This inclusion is an ingredient
  in simplicity proofs for $\Phi$ that use the Iwasawa criterion:
  
  For $\Phi = \PSL2{\FF_q}$ the necessary commutators are computed
  in the proof of \cite[Theorem 4.4, page~23]{MR1189139}. The case where
  $\Phi = \PSU3{\FF_{t^2}|\FF_t}$ is covered by \cite[Proof of II.10.13,
  page~244]{MR0224703}. For $\Phi = \Sz{2^s}$ the assertion follows from
  the commutator formula in \cite[page~205]{MR1189139}, and for
  $\Phi = \Ree{3^r}$ the three commutator
  formulas in \cite[\S~5, pages 36/37]{MR2730410} yield the assertion.
\end{proof}

\begin{rema}
  The references in the proof of~\ref{simpleMoufangSetCommutators}
  yield the following sharper conclusion: %
  for every $y\in X\setminus \{x\}$ 
  there exists an element $\phi \in \Phi_{x,y}$ such that $\Delta_x$
  is equal to the set $\left[\Delta_x,\phi\right] := %
  \set{ [\delta, \phi]}{ \delta \in \Delta_x}$ 
  of commutators. See also the proof of \cite[2.11 b)]{MR0467511}.
\end{rema}

\begin{prop}\label{PropSimple}
  If\/ $\lpg$ is simple then $\gp$ is a perfect central extension
  of\/~$\lpg$, i.e.\ $\gp$ coincides with its commutator group~$\gp'$,
  and\/~$\gp_ {[B_\infty]}$ is isomorphic to a quotient of the Schur
  multiplier of\/~$\lpg$.
\end{prop}

\begin{proof}
  If $z\in B_\infty$, $\tau\in\tg{z}$ and
  $\gamma\in\gp_z$, then $\gamma^{-1}\tau\gamma\in\tg{z}$ and 
  $[\tau,\gamma] =
  \tau^{-1}\gamma^{-1}\tau\gamma\in\tg{z}$.
  Since $\tg{z}$ acts regularly on~$B_\infty\setminus\{z\}$, every
  element of $\tg{z}$ is determined by its action
  on~$B_\infty$, i.e.\ by its image in~$\lpg$.
  By \ref{simpleMoufangSetCommutators} every element of $\tg{z}$
  is a product of elements in $\tg{z}$ that are commutators.
  Hence $\tg{z} \le G'$ for every $z\in B_\infty$, and therefore
  $G'=G$.
  
  The kernel~$\gp_{[B_\infty]}$ of the action on~$B_\infty$ is the
  center of $\gp$, so the perfect group $G$ is a central extension of
  $\lpg = G / \gp_{[B_\infty]}$. Therefore $\gp_{[B_\infty]}$ is
  isomorphic to a quotient of the Schur multiplier of $\lpg$; see
  \cite[2.1.7]{MR1200015}, \cite[V.23.3, page 629]{MR0224703}, or
  \cite[33.8\,(4) p.\,169]{MR1777008}.
\end{proof}

\section{Sharply two-transitive groups}

\begin{prop}\label{PSL23}
  If $q\le 3$ then\/ $\UU$ is the hermitian unital of order~$q$.
\end{prop}

\begin{proof}
  Every unital of order $2$ is isomorphic to the hermitian one, see
  e.g.\ \cite[10.16]{MR1189139}.
  Now let $q=3$. Since $\lpg\le \operatorname{S}_4$ is generated by
  elements of order $3$, we have $\lpg = \Alt4 \cong \PSL2{\FF_3}$; in
  particular, $\lpg$ is sharply two-transitive. 
  By~\cite[2.3]{MR3090721} the center $Z=\gp_{[B_\infty]}$ has even order,
  so there exists a central involution~$\zeta$
  in~$\gp$. %

  The product~$\gp'Z$ induces the commutator group $(\lpg)'\cong\Cg2^2$.
  Thus $\gp'Z$ has index~$3$, and~$\gp'$ acts transitively
  on~$B_\infty$. For $z\in B_\infty$, the translation group~$\tg{z}$
  is not contained in~$\gp'$. We obtain $\gp=\gp'\tg{z} \gp' =
  \gp'\tg{z}$, and $\gp'$  has index~$3$ in~$\gp$. %
  This means that $Z\le\gp'$, and~$\gp$ is a 
  covering group of~$\Alt4$. %
  Then $\gp\cong \SL2{\FF_3}$ by~\cite[2.12.5]{MR1200015}.
  This group acts regularly on $U\setminus B_\infty$,
  see~\cite[3.5]{MR3090721}.

  We verify that~$\UU$ is obtained by the construction described
  in~\cite[2.1]{MR3533345}. %
  The central involution~$\zeta$ does not fix any point apart from
  those on~$B_\infty$. Therefore, the point set $U\setminus B_\infty$
  is partitioned by fixed blocks of~$\zeta$; these are obtained as the
  blocks joining $x\in U\setminus B_\infty$ with its image
  under~$\zeta$. %
  The group~$\gp$ acts on this set of fixed blocks. There are~$6$ such
  blocks, and at least one of them is fixed by a subgroup~$S$ of
  order~$4$ in~$\gp$. We pick a point~$a$ on that block and identify
  the elements of~$\gp$ with the affine points via
  $\gamma\mapsto a^\gamma$. Then the block in question is~$S$.
  
  There are~$4$ blocks through~$a$ that join~$a$ to points
  on~$B_\infty$, their intersections with $U\setminus B_\infty$ are
  identified with the Sylow $3$-subgroups (viz., the translation
  groups) in~$\gp$. Let~$D$ be any one of the remaining~$4$ blocks
  through~$a$. Then~$D$ is not stabilized by any translation, and not
  stabilized by~$\zeta$. As $\SL2{\FF_3}$ contains only one
  involution, we infer that the stabilizer of~$D$ in~$\gp$ is
  trivial. Therefore, the set
  $\mathcal{D} \coloneqq \smallset{D\delta^{-1}}{\delta\in D}$
  consists of~$4$ different blocks through~$a$.

  It has been proved in~\cite[3.3]{MR3533345} that the subgroup~$S$
  and the set~$\mathcal{D}$ are unique, up to conjugation. The points at
  infinity are the centers of translations. Therefore each such point
  is incident with those blocks whose points outside~$B_\infty$ form an
  orbit under the corresponding translation group.
  This completes the proof that~$\UU$ is isomorphic to the hermitian
  unital~$\UU_{\mathcal{H}_3}$, see~\cite[3.3]{MR3533345}.  
\end{proof}

Now we determine certain central extensions of finite sharply two-transitive
permutation groups; the following result is a variation of
\cite[Lemma 1.1]{MR533094} that is suitable for our purpose.

\begin{theo}\label{centralext}
Let $(G,X)$ be a finite sharply two-transitive permutation group
with $|X| >1$ and let $p$ be the prime dividing $|X|$. If $E$ is
a central extension of $G=E/Z$ by a group $Z$ of order $p$,
then $E$ splits over $Z$ (as a direct product $E=Y\times Z$ with
$Y\cong G$), or we have one of the following:
\begin{enumerate}
\item $|X| = 2 = |G|$ and $E$ is cyclic of order $4$.
\item $|X| = 4$, $G = \Alt4 \cong \PSL 2 {\FF_3}$ and $E\cong \SL 2 {\FF_3}$.
\item $|X| = p^2 \in \{ 3^2, 5^2, 7^2, 11^2 \}$ and $E = P \rtimes H$
where $P$ is the Heisenberg group of order $p^3$ and $H$ is isomorphic
to $Q_8$, $\SL 2 {\FF_3}$, $2^- \Sym4$ or $\SL 2 {\FF_5}$, respectively.
\end{enumerate}
\end{theo}

We describe the groups in item (c).
The Heisenberg group of order $p^3$ consists of all unipotent upper
triangular matrices in $\GL 3 {\FF_p}$. By $Q_8$ we denote the quaternion
group of order $8$, and $2^- \operatorname{S}_4$ is the binary octahedral group,
i.e.\ the double cover of $\operatorname{S}_4$ containing
just one involution, see \cite[3.2.21, p.~301]{MR648772} or
\cite[XII.8.4]{MR662826}; this double cover is isomorphic to the normalizer of
$\SL2 {\FF_3}$ in $\SL2{\FF_9}$. The extension groups $E$ in item~(c) do not
split over $Z$ since $P$ is not abelian.

\begin{proof}[Proof of\/ \ref{centralext}]
  It is well known that $|X|=p^n$ is a power of a prime $p$ and that
  the Sylow $p$-subgroup of $G$ is an elementary abelian normal
  subgroup of order $p^n$ in $G$; see e.g.\ \cite[7.3.1]{MR1357169} or
  \cite[8.4]{MR0237627} or \cite[XII.9.1]{MR662826}.

  Let $P$ be a Sylow $p$-subgroup of $E$. Then $|P|= p^{n+1}$ and
  $Z\le P$; moreover $P/Z$ is the regular normal subgroup of $G$,
  hence $P$ is normal in $E$. Each point stabilizer (or Frobenius
  complement) $G_x$ has order $p^n -1$, and its pre-image $E_x\le E$
  has order $(p^n-1)p$.  The group $E_x$ splits as a direct product
  $H\times Z$ with $H\cong G_x$ by the abelian (in fact, central) case
  of the Schur--Zassenhaus theorem; see \cite[9.1.2 or
  11.4.12]{MR1357169} or \cite[10.3]{MR0237627} or \cite[I.17.5,
  page~122]{MR0224703}. Then
  \[
    E = P \rtimes H
  \]
  and $H$ acts (by conjugation) sharply transitively on the set of
  non-trivial elements of~$P/Z$.

  If $H$ is trivial, then $|X| = 2 = |G|$, and $E$ splits or is cyclic
  as in item (a). From now on let $|H|>1$. Then $C_P (H) / Z$ is a
  proper $H$-invariant subgroup of $P/Z$, hence trivial. This means
  that $C_P (H) =Z$.  If $P$ is abelian, then
  $P= C_P (H) \times [P,H] = Z \times [P,H]$, see
  \cite[5.2.39]{MR569209} or \cite[10.1.6]{MR1357169} or
  \cite[III.13.4, p.\,350]{MR0224703}, and then
  $E = P \rtimes H = Z \times ([P,H] \rtimes H)$ splits over $Z$.

  Now let $P$ be non-abelian. Then $P'$ is a non-trivial subgroup of
  $Z$, as $P/Z$ is (elementary) abelian, hence $P' = Z$. The center of
  $P$ yields a proper $H$-invariant subgroup of $P/Z$; this subgroup
  is trivial, hence $Z$ is the center of $P$ (and the Frattini
  subgroup is $\Phi(P) = P^p P' = Z$). Thus $P$ is an extraspecial
  $p$-group.

  The commutator map gives a non-zero symplectic form $f$ on $P/Z$
  with values in the prime field $F_p$, and $f$ is not degenerate,
  hence $n=2m$ is even; see \cite[p.~140]{MR1357169} or
  \cite[III.13.7, p.\,353]{MR0224703}. The automorphism group $\gal H$
  induced by $H$ on $P$ has trivial intersection with the group of
  inner automorphisms of $P$ and acts trivially on $Z$, hence
  $H\cong \gal H$ is isomorphic to a subgroup of the symplectic group
  $\Sp{2m} {\FF_p}$ by Winter \cite[Theorem 1 or (3A)
  p.~161]{MR297859}.

  First we assume that the permutation group $(G,X)$ is of type~I %
  (in the notation of~\cite[XII.9.2]{MR662826}), which entails that
  $H\le \operatorname{\Gamma L}(1, \FF_{p^n}) = \GL 1 {\FF_{p^n}}
  \rtimes \Aut{\FF_{p^n}}$; see \cite[XII.9.2]{MR662826}. In another
  terminology, this means that the corresponding nearfield (with
  multiplicative group $H$) is a Dickson nearfield, compare
  \cite[p.\,834]{MR3667125}.  The cyclic group
  $H\cap \GL 1 {\FF_{p^n}} \le \FF_{p^n}^\ast$ has order at least
  $(p^n-1)/n$. If this cyclic group is reducible on $\FF_p^n$, then it
  is contained in a proper subfield of $\FF_{p^n}$, hence
  $(p^n-1)/n \le p^{n/2} -1$ and therefore $p^{n/2} +1 \le n$; if
  $H\cap \GL 1 {F_{p^n}}$ is irreducible, then its order divides
  $p^{n/2} +1$ by \cite[Cor.\,2]{MR297859} or \cite[Satz\,9.23,
  p.\,228]{MR0224703} as $H\le \Sp n {\FF_p}$. In both cases we have
  $2^{n/2} -1 \le p^{n/2} -1 \le n$, which is false for $n \ge 6$. If
  $n=4$ then $p=2$ and $|H|=15$, hence $H$ is cyclic and irreducible,
  but $15$ does not divide $2^2 +1$. %
  As $n=2m$ is even, there only remains the case where $n=2$, and
  $p\in\{2, 3\}$ follows.

  If $p=2$ then $|X| = 4$ and $G = \Alt4 \cong \PSL 2 {\FF_3}$;
  moreover, $E$ is a covering group of $\Alt4$ since $Z=P' \le E'$,
  hence $E\cong \SL 2 {\FF_3}$ by \cite[2.12.5]{MR1200015}, as in
  item~(b). For $p=3$ we have $|X| = 9$ and $|H|=8$. Each involution
  $h\in H$ induces on $P/Z \cong \FF_3^2$ a diagonalizable linear
  transformation $\gal h$ without eigenvalue $1$, hence
  $\gal h = - \mathrm{id}$.  Thus $H$ contains just one involution,
  and $H$ is cyclic or $H\cong Q_8$ (these two possibilities
  correspond to the two nearfields of order $9$, one of them being the
  field $\FF_9$). The cyclic case is ruled out because
  $\Sp 2 {\FF_3} = \SL 2{\FF_3}$ contains no element of order
  $8$. Thus $H\cong Q_8$ as in the first case of item~(c).

Now we assume that $(G,X)$ is not of type I. Then $n=2$ and there are
just seven possibilities for the isomorphism type of $H$, with $p\in
\{5, 7, 11, 23, 29, 59\}$: 
see \cite[XII.9.4]{MR662826} or \cite[20.3]{MR0237627} or \cite[2.4]{MR3667125}.
This rephrases a famous result of Zassenhaus, which says that there
are only seven finite nearfields which are not Dickson nearfields.
The condition $H\le \Sp 2 {\FF_p} = \SL 2 {\FF_p}$ excludes
four of these seven possibilities (those where $H$ contains  central
elements other than $\pm\mathrm{id}$), see \cite[XII.9.4,
XII.9.5]{MR662826} or \cite[2.4]{MR3667125}. This leads to the three
cases for $H$ in item (c) with $p\in \{5, 7, 11\}$. 

\goodbreak
For all our odd primes $p$, the extraspecial group $P$ of order $p^3$
has exponent $p$: otherwise the exponent is $p^2$ and some non-trivial
element of $P/Z$ is 
fixed by every automorphism of $P$ by \cite[Cor.~1]{MR297859}; this is
a contradiction to the action of $H$ on $P/Z$. Therefore $P$ is
isomorphic to the Heisenberg group of order $p^3$, see
\cite[5.5.1]{MR569209} or \cite[p.~355]{MR0224703}. 
\end{proof}

\begin{theo}\label{sharplyTwoTrs}
  If\/~$\lpg$ is sharply two-transitive on~$B_\infty$, then $q\le3$ %
  and\/~$\UU$ is the hermitian unital of order $q$.
\end{theo}

\begin{proof}
  By Proposition \ref{PSL23} it suffices to show that $q\le 3$. Thus
  we assume that $q>3$ and aim for a contradiction.
 
  The degree $q+1$ of $\lpg$ is a power of some prime $r$, say
  $q+1=r^n$. %
  By \cite[3.1]{MR3090721} the kernel $G_{[B_\infty]}$ is the center
  $Z$ of $G$, and $r$ divides $|Z|$ by \cite[2.3]{MR3090721}. Thus we
  can choose a subgroup $U$ of index $r$ in $Z$; then $G/U$ is a
  central extension of $G^\dagger$ by the group $Z/U$ of order
  $r$. Such an extension $G/U$ does not split: if
  $G/U = Z/U \times Y/U$ then $Y$ contains all Sylow $s$-subgroups of
  $G$ with $s\ne r$, hence all translation groups $\Gamma_{[x]}$ with
  $x\in B_\infty$; thus $Y=G$, which is a contradiction to
  $|Z/U| = r$.

  Theorem \ref{centralext} implies that $n=2\ne r$ and that the Sylow
  $r$-subgroup of $G/U$ is not abelian (and more, as in item (c), but
  we do not need more).  Let $R$ be a Sylow $r$-subgroup of $G$ and
  let $H:= \Gamma_{[o]}$.  Then $RZ/Z$ is the regular normal subgroup
  of $G^\dagger$, and $R$ is characteristic in $RZ$, which is normal
  in $G$; hence $R$ is normal in $G$.  The group $RHR= RH$ contains
  all conjugates of $H$ in $G$, hence $G=RH = R\rtimes H$. Thus
  $Z = G_{[B_\infty]} = G_{o, \infty} = (R_o H)_\infty = (R\cap
  Z)H_\infty = R\cap Z$, which gives $Z\le R$.  The group $R$ is not
  abelian, but $R/Z$ is abelian and has order $r^2$; thus $Z$ is the
  center of $R$. %
  Now a (special case of a) result of Wiegold says that $|R'|$ divides
  $r$; see \cite[Theorem 2.1]{MR0179262}, \cite[p.~261]{MR648772},
  \cite[Lemma 3.1.1, p.~113]{MR1200015} or \cite[page 637]
  {MR0224703}.  We claim that $R'=Z$. Otherwise we can choose $U$ as
  above with $R' \le U < Z$, and then $R/U$ is an abelian Sylow
  $r$-subgroup of $G/U$, contrary to Theorem \ref{centralext}.

  Thus $R' = Z$ has order $r$, and $R$ is an extraspecial group of
  order $r^3$. Since $r\ne 2$ the group $H= \Gamma_{[o]}$ contains an
  involution $\alpha$ inducing inversion on $R/R'=R/Z$, hence $\alpha$
  fixes each subgroup between $Z$ and $R$.

  Each subgroup of order~$r^2$ is normal in~$R$ with abelian quotient,
  and thus contains $R'=Z$. %
  As the group~$H$ acts transitively on the set of non-trivial
  elements of~$R/R'$, it also acts transitively on the set of
  subgroups of order~$r^2$ in~$R$. If $S$ is one of those subgroups
  then~$R$ acts transitively on the set of non-central subgroups of
  order~$r$ in~$S$. There are~$r$ such subgroups, and the
  involution~$\alpha$ (which leaves~$S$ invariant) fixes at least one
  of them, say~$T$.

  The number of points not on~$B_\infty$ is
  $q^3+1-(q+1) = r^2(r^2-1)(r^2-2)$, and not divisible by $r^3 = |R|$.
  Therefore, there exists some subgroup of order~$r$ fixing at least
  one point~$x$ not on~$B_\infty$. That subgroup is not contained in the
  center because the latter acts semi-regularly on $U\setminus B_\infty$,
  see~\cite[1.7]{MR3090721}. We have noted in the previous paragraph
  that the non-central subgroups of order~$r$ form a single conjugacy
  class in~$HR$. Thus the group~$T$ fixes some affine point~$x$.  Then
  $T = T^\alpha$ fixes also~$x^\alpha$, and the point~$o$
  where~$B_\infty$ meets the block joining~$x$ and~$x^\alpha$. This
  contradicts the fact that~$T$ induces a subgroup of order~$r$ in the
  regular normal subgroup on~$B_\infty$.
\end{proof}

\goodbreak
\section{Unitary groups}

\begin{lemm}\label{conjugacy}
  Let\/ $r$ be  prime power, and let\/ $d$ be a divisor of~$r+1$. Then
  the following hold:
  \begin{enumerate}
  \item Every element of order~$d$ in $\GL{3}{\FF_{r^2}}$
    is diagonalizable over~$\FF_{r^2}$.
  \item If~$A$ is an element of order~$d$ in $\SU{3}{\FF_{r^2}|\FF_r}$
    then the characteristic polynomial of~$A$ is
    $X^3-t_AX^2+\gal{t_A}X-1$, where~$t_A$ is the trace of~$A$.
  \item Two elements of order~$d$ in~$\SU{3}{\FF_{r^2}|\FF_r}$ are
    conjugates under~$\SU{3}{\FF_{r^2}|\FF_r}$ if, and only if, they
    have the same trace. %
  \end{enumerate}
\end{lemm}
\begin{proof}
  (a) %
  Let $A\in\GL{3}{\FF_{r^2}}$ be an element of order~$d$. The minimal
  polynomial of~$A$ then divides $X^{r+1}-1$, and every characteristic
  root is a root of that polynomial. %
  These roots lie in~$\FF_{r^2}$ because %
  $r+1$ divides the order of the multiplicative group of~$\FF_{r^2}$.
  As the minimal polynomial has only simple roots, the matrix~$A$ is
  diagonalizable in~$\GL{3}{\FF_{r^2}}$.

  (b) %
  Now assume $A\in\SU{3}{\FF_{r^2}|\FF_r}$. %
  Let $\lambda$ be one of the characteristic roots of~$A$, %
  then $\gal\lambda\lambda = \lambda^r\lambda = 1$. 
  In the characteristic polynomial
  $\det(X\cdot\id-A) = X^3+c_2X^2+c_1X+c_0$, the constant~$c_0$ equals
  $-\det{A} = -1$. The coefficient $c_2$ equals~$-t_A$, where~$t_A$ is
  the trace of~$A$. Expanding the product of the linear factors, we
  obtain $t_A = -c_2$ as the sum $\lambda_0+\lambda_1+\lambda_2$ of
  all characteristic roots of~$A$.  The coefficient~$c_1$ is obtained
  as
  $\lambda_0\lambda_1+\lambda_1\lambda_2+\lambda_2\lambda_0 =
  \lambda_2^{-1}+\lambda_0^{-1}+\lambda_1^{-1} =
  \gal{\lambda_2}+\gal{\lambda_0}+\gal{\lambda_1} = \gal{t_A}$. %

  (c) %
  Let~$A$ and~$B$ be elements of order~$d$
  in~$\SU{3}{\FF_{r^2}|\FF_r}$. Clearly $t_A = t_B$ holds if~$A$
  and~$B$ are conjugates. %
  Conversely, assume $t_A=t_B$.  We have seen above that~$A$ and~$B$
  have the same characteristic polynomial. Therefore, they are
  conjugates in~$\GL{3}{\FF_{r^2}}$. %
  According to~\cite[I,\,3.5, III,\,3.22]{MR0268192} %
  (or~\cite[Case~A\,(ii), p.\,34]{MR0150210} or %
  \cite[Lemma~5 with remarks on p.~12]{MR0139651}) %
  they are also conjugates in the unitary group
  $\U{3}{\FF_{r^2}|\FF_r}$.

  Finally, the group $\U{3}{\FF_{r^2}|\FF_r}$ contains diagonal
  elements of arbitrary determinant
  in~$\smallset{s\in\FF_{q^2}}{s\gal{s}=1}$. As such diagonal matrices
  centralize each other diagonal matrix, we can adapt the conjugating
  element of $\U{3}{\FF_{r^2}|\FF_r}$ in such a way that the
  conjugation is achieved by an element of $\SU{3}{\FF_{r^2}|\FF_r}$.
\end{proof}

The following lemma is a consequence of results 
in~\cite[Thm.\,1.6, Thm.\,1.3]{MR3247775}; we give a direct proof for
the reader's convenience.  
  
\begin{lemm}\label{productOfRootElements}
  Let\/ $r=2^e$ and let $A\in \SU3{\FF_{r^2}|\FF_r}$ be non-central
  with $A^{r+1} =1$.  Then $A^2$ is the product of two elements of\/
  $\SU3{\FF_{r^2}|\FF_r}$ with orders dividing $4$. %
\end{lemm}

\begin{proof}
  We use coordinates such that the hermitian form is described by
  $x_0\gal{y_2}+x_1\gal{y_1}+x_2\gal{y_0}$. 
  The element %
  \(%
  J \coloneqq %
  \left(%
    \begin{smallmatrix}
      1 & 0 & 0 \\
      0 & 1 & 0 \\
      1 & 0 & 1
    \end{smallmatrix} %
  \right) \in \SU3{\FF_{r^2}|\FF_r} %
  \)   %
  is an involution, and %
  \(%
  F_{u,v} \coloneqq  %
  \left(%
    \begin{smallmatrix}
        1 & u & v \\
        0 & 1 & \gal{u}\\
        0 & 0 & 1
    \end{smallmatrix} %
  \right) %
  \) %
  belongs to $\SU3{\FF_{r^2}|\FF_r}$ if $v+\gal{v} = u\gal{u}$. Note
  also that $F_{u,v}^4=\id$, and~$F_{u,v}^2=\id$ holds if $u=0$ (then
  $v\in\FF_{r}$). %
  The product
  $
    JF_{u,v} =  %
    \left(
      \begin{smallmatrix}
        1 & u & v \\
        0 & 1 & \gal{u}\\
        1 & u & v+1
      \end{smallmatrix}
    \right)
  $
  has trace $v+1$, and its characteristic polynomial is
  $X^3+(v+1)X^2+(\gal{v}+1)X+1$.

  Let $t_A$ be the trace of the given matrix $A$ and put $v\coloneqq t_A+1$.
  The norm map
  $N\colon \FF_{r^2}\to\FF_r\colon x\mapsto x\gal{x} = x^{r+1}$ is
  surjective, hence we find $u\in\FF_{r^2}$ such that
  $u\gal{u} = v+\gal{v}$.
  We abbreviate $F \coloneqq F_{u,v}$ and infer from \ref{conjugacy}
  that $JF$ and the diagonalizable matrix $A$ have the same
  characteristic polynomial, hence also the same set of eigenvalues.
  If $JF$ is diagonalizable, then $JF$ has the same order as $A$, and
  \ref{conjugacy} implies that $A$ is conjugate to $JF$ in
  $\SU3{\FF_{r^2}|\FF_r}$. %
  Now $A^2$ is conjugate to $(JF)^2 = JFJF = (JFJ^{-1})F$. 
  
  It remains to consider the case where $JF$ is not diagonalizable. Then
  the characteristic polynomial has a root $\lambda$ with multiplicity $2$
  (not $3$ since $A$ is not central), and $A$ is conjugate to the
  diagonal matrix $\operatorname{diag}(\lambda, \lambda, \lambda^{-2})$
  where $N(\lambda) = \lambda^{r+1} = 1 \ne \lambda ^3$.
  Thus $JF$ is similar %
  (i.e.\ conjugate in $\GL 3 {\FF_{r^2}}$) to its Jordan normal form
   \[
    \left(
      \begin{matrix}
        \lambda & 0 & 0 \\
        1 & \lambda & 0 \\
        0 & 0 & \lambda^{-2}
      \end{matrix}
    \right),  %
  \]
  hence $(JF)^2$ is similar to
  $\operatorname{diag}(\lambda^2, \lambda^2, \lambda^{-4})$ which is
  similar to $A^2$.  The matrix $(JF)^2= JFJF= (JFJ^{-1})F$ is
  conjugate to $A^2$ in $\SU3{\FF_{r^2}|\FF_r}$ by \ref{conjugacy}.
\end{proof}
 
\begin{rema}\label{nonCentralNeeded}
  The assumption that $A$ is not central is needed
  in~\ref{productOfRootElements}. %
  Indeed, for any field~$F$ of characteristic two, non-trivial central
  elements of $\GL{n}{F}$ are never products of two elements in Sylow
  $2$-subgroups. In fact, a non-trivial central element is of the form
  $u\,\id$ with $u\in F$. The elements of Sylow $2$-subgroups are
  unipotent (i.e.\ they have~$1$ as their only characteristic
  root). If the product of unipotent elements $S,T$ equals~$u\,\id$
  then $S = uT^{-1}$ is a unipotent element with characteristic
  root~$u$, so $u=1$ and the product is trivial, indeed.
\end{rema}

\begin{theo}\label{excludeSU}
  The little projective group~$\lpg$ is not isomorphic to
  $\PSU3{\FF_{r^2}|\FF_r}$, for any~$r$.
\end{theo}
\begin{proof}
  If~$\lpg$ is isomorphic to $\PSU3{\FF_{r^2}|\FF_r}$ then the
  translation groups are the root subgroups, i.e.\ the (Sylow)
  subgroups of order~$r^3$ in $\PSU3{\FF_{r^2}|\FF_r}$. In particular,
  we have $q=r^3$.
  For $r=2$ we have $q=8$, and~$\lpg$ is (isomorphic to) the sharply
  two-transitive group
  $\PSU3{\FF_{4}|\FF_2} \cong \Q8 \ltimes \FF_3^2$; this is excluded
  by~\ref{sharplyTwoTrs}.

  From now on, let $r>2$. The group $\PSU3{\FF_{r^2}|\FF_r}$ is
  perfect, and~$\gp$ is a perfect central extension
  of~$\PSU3{\FF_{r^2}|\FF_r}$, see \ref{PropSimple} or \cite[3.1]{MR3090721}.
   For the case
  at hand, we know that $\SU3{\FF_{r^2}|\FF_r}$ is the universal cover
  of $\PSU3{\FF_{r^2}|\FF_r}$, see~\cite[Thm.\,2]{MR0338148}. %
  So we assume that $\SU3{\FF_{r^2}|\FF_r}$ acts (not necessarily
  faithfully) on the unital~$\UU$ such that the root subgroups induce transitive
  groups of translations with center on~$B_\infty$. %

  Assume first that $r$ is odd, and let $2^a$ be the highest power
  of~$2$ dividing $|U\setminus B_\infty|  = (r^3+1)r^3(r^3-1)$. Then
  $2^{a}$ divides $(r^3+1)(r-1)$ and $2^{a+1}$ divides
  $(r^3+1)r^3(r-1)(r+1) = |\SU3{\FF_{r^2}|\FF_r}|$. So some
  point in $ U\setminus B_\infty$ is fixed by some involution~$\gamma
  \in \SU3{\FF_{r^2}|\FF_r}$.
  We use coordinates such that the hermitian form defining
  $\SU3{\FF_{r^2}|\FF_{r^2}}$ is given by
  $x_0\gal{y_2} + x_1\gal{y_1} + x_2\gal{y_0}$.
  Then the matrices $ \left(
    \begin{smallmatrix}
      1 & 0 & 0 \\
      1 & 1 & 0 \\
      -1/2 & -1 & 1 
    \end{smallmatrix}\right)$
  and $  \left(
    \begin{smallmatrix}
      1 & -4 & -8 \\
      0 & 1 & 4 \\
      0 & 0 & 1
    \end{smallmatrix}\right)$ belong to root groups of
  $\SU3{\FF_{r^2}|\FF_r}$, and their
  product $\left(
    \begin{smallmatrix}
      1 & -4 & -8 \\
      1 & -3 & -4 \\
      -1/2 & 1 & 1 
    \end{smallmatrix}\right)$ is an involution (and represents an
  involution in $\PSU3{\FF_{r^2}|\FF_r}$).
  All involutions in $\SU3{\FF_{r^2}|\FF_r}$ are conjugate by~\ref{conjugacy},
  hence $\gamma$ is a product of two root elements
  and does not
  fix any point outside $B_\infty$; this is a contradiction.

  Therefore $r$ is even. %
  Let~$p$ be a prime dividing $r+1$, and let $m$ be
  the largest integer such that~$p^m$ divides $r+1$. Then $p$ is odd
  (because $r$ is even), and $p^{2m}$ divides
  $(r^3+1)r^3(r^2-1) = |\SU3{\FF_{r^2}|\FF_r}|$.

  If $p>3$ then $p$ does not divide $r^2-r+1$, and $p^{m+1}$ does not
  divide $|U\setminus B_\infty| = ({r^3+1})r^3({r^2+r+1})({r-1})$. So
  there exists at least one orbit whose length is not divisible
  by~$p^{m+1}$, and there exists an element~$\gamma$ of order~$p$ in
  the stabilizer of some point not in~$B_\infty$. If~$\gamma$ is not
  central in $\SU{3}{\FF_{r^2}|\FF_r}$ then $\gamma^2$
  is a product of two root elements (see~\ref{productOfRootElements})
  and does not fix any point outside~$B_\infty$.
  So~$\gamma$ is a central element of order~$p>3$ in
  $\SU{3}{\FF_{r^2}|\FF_r}$, contradicting the fact 
  that the center of $\SU{3}{\FF_{r^2}|\FF_r}$ has order~$3$ or is
  trivial. 

  There remains the case where $p=3$ is the only prime divisor
  of~$r+1$. Then $r+1 = 2^d + 1 = 3^m$ for positive integers $d$
  and~$m$. We infer that $r=2^d \in\{2,8\}$, see e.g.\ \cite[Lemma
  19.3]{MR0237627}; this is an old result of Levi ben Gerson from
  1343, %
  see~\cite[\S\,4, pp.\,169\,ff]{Chemla-Pahaut}.  Since $r>2$ we have
  $r=8$ and $m=2$.  Then $3^{3}=3^{m+1}$ is the highest power of~$3$
  dividing
  $|U\setminus B_\infty| = ({r^3+1})r^3({r^2+r+1})({r-1}) =
  2^9\cdot3^3\cdot7\cdot19\cdot73$ but $3^5=3^{2m+1}$ divides
  $|\SU{3}{\FF_{64}|\FF_8}| = (r^3+1)r^3(r^2-1) =
  2^9\cdot3^5\cdot7\cdot19$.  %
  We now find an element $\gamma$ of order~$3$ in the stabilizer of a
  point not in~$B_\infty$.  %
  If~$\gamma$ is not central in $\SU{3}{\FF_{64}|\FF_8}$ then
  $\gamma =\gamma^{-2}$ is a product of two root elements
  (see~\ref{productOfRootElements}) and does not fix any point
  outside~$B_\infty$. So~$\gamma$ is a central element of order~$3$ in
  $\SU{3}{\FF_{64}|\FF_8}$ and fixes every point in~$\UU$,
  see~\ref{applyRigidity}. %
  This means that $\SU{3}{\FF_{64}|\FF_8}$ induces on~$\UU$ a group
  isomorphic to $\PSU{3}{\FF_{64}|\FF_8}$, of order
  $(r^3+1)r^3(r^2-1)/3 = (8^3+1)8^3(8^2-1)/3 =
  2^9\cdot3^4\cdot7\cdot19$. Since $3^4$ does not divide
  $|U\setminus B_\infty|$ we still find an element of order~$3$ in the
  stabilizer of a point not on $B_\infty$, and reach a contradiction
  using \ref{productOfRootElements} again.
\end{proof}

\section{Suzuki groups and Ree groups}

\begin{theo}\label{SuzukiSemiReg}
  If\/ $\lpg$ is a Suzuki group then $q\ge 2^6$ and $\gp = \lpg$, and~$\gp$
  acts semi-regularly on $U\setminus B_\infty$.
\end{theo}
\begin{proof}
  We have $\lpg = \Sz{2^s}$ for some odd integer $s\ge 1$, and
  the unital has order $q = 2^{2s}$. %
  The smallest Suzuki group $\Sz{2} \cong \AGL1{\FF_5}$ is sharply
  two-transitive, and excluded by~\ref{sharplyTwoTrs}. %
  
  The Schur multiplier of~$\Sz{2^3}$ is elementary abelian of
  order~$4$, see~\cite{MR0193141}, cf.~\cite[4.2.4]{MR2562037} and
  \cite[7.4.2]{MR1200015}. %
  If~$\zeta$ is a central involution in~$\gp$ then $\zeta$ acts
  trivially on $B_\infty$, and joining any point~$x$ with~$x^\zeta$
  gives a block~$B$ fixed by~$\zeta$. If that block does not
  meet~$B_\infty$ then~$\zeta$ fixes at least one of the $q+1 = 65$
  points on~$B$. This contradicts~\ref{applyRigidity}. So~$B$ contains
  a point~$z$ of~$B_\infty$. Then there exists a translation~$\tau$
  with center~$z$ such that $x^\zeta = x^\tau$. The translations have
  order dividing $4$, hence $\tau\zeta$ is an element of order~$2$
  or~$4$ fixing~$x$. %
  If $\tau$ has order~$4$ then $(\tau\zeta)^2 = \tau^2$ is non-trivial
  translation fixing~$x$.  This is impossible, so~$\tau$ is an
  involution. %
  The automorphisms~$\zeta\tau$ and~$\tau$ induce the same action
  on~$B_\infty$. In particular, the involution~$\zeta\tau$ fixes no
  point on~$B_\infty$ apart from~$z$. %
  For each point $y\in U\setminus B_\infty$, the block joining~$y$
  and~$y^{\tau\zeta}$ is fixed by~$\tau\zeta$, and meets~$B_\infty$ in
  a fixed point of~$\tau\zeta$; that point has to be~$z$. This means
  that $\tau\zeta$ fixes every block through~$z$, and is a translation
  with center~$z$. Now $\zeta = \tau(\tau\zeta) \in \tg{z}$ is a
  translation fixing every point on~$B_\infty$. This contradicts the
  fact that a non-trivial translation fixes only one point. So
  $\gp = \lpg$ holds if~$\lpg=\Sz{2^3}$.
  
  If $\lpg = \Sz{2^s}$ with $s>3$ then $\gp = \lpg$ because the
  Schur multiplier is trivial; see~\cite{MR0193141},
  cf.~\cite[4.2.4]{MR2562037} and \cite[7.4.2]{MR1200015}. %
  Thus we have $\gp=\lpg = \Sz{2^s}$ for $s\ge 3$.
  Consequently, each element of order~$2$ or~$4$ in~$\gp$ is a
  translation.
  Apart from the elements of order~$4$, every element in~$\Sz{2^s}$ is
  strongly real, i.e.\ a product of two involutions; see e.g.~\cite[24.7,
  24.6]{MR572791}. In particular, every non-trivial element is the product of two
  translations (viz., elements of order dividing~$4$), and does not fix
  any point in $U\setminus B_\infty$. So the action
  of~$\gp$ on $U\setminus B_\infty$ is semi-regular. %
\end{proof}

The following result is contained in \cite[2.6]{MR533094}; we give a 
more detailed proof.

\begin{lemm}\label{productsRee}
  In the Ree group $\Ree{r}$ with $r=3^{2e+1}\ge 3$, every element of prime
  order is the product of two elements with orders dividing $9$.
\end{lemm}

\begin{proof}
  All involutions in $\Ree{r}$ are conjugate (also for $r=3$), so each
  of them is contained in a subgroup isomorphic to
  $\Ree3\cong\PgL2{\FF_8} \cong \SL2{\FF_8}\rtimes\Cg3$. The
  factorization $%
  \left(
    \begin{smallmatrix}
      1 & u+1 \\
      0 & 1 
    \end{smallmatrix}\right)
  =
  \left(
    \begin{smallmatrix}
      1 & 1 \\
      1 & 0 
    \end{smallmatrix}\right)
  \left(
    \begin{smallmatrix}
      0 & 1 \\
      1 & u 
    \end{smallmatrix}\right) %
  $ %
  in $\SL 2 {\FF_8}$, where $u\in \FF_8$ satisfies $u^3 + u +1 =0$,
  shows that every involution is the product of an element of order
  $3$ with an element of order $9$ (it is also the product of two
  elements of order $9$, see \cite[Case\,(6), p.\,429]{MR3090721}). 

  The root elements of $\Ree{r}$ have orders dividing $9$; thus it
  remains to consider elements with prime order $p> 3$.  We have
  \[
    |\Ree{r}| =  (r^3+1)r^3 (r -1) 
    =r^3 (r^2 -1) (r + \sqrt{3r} +1)(r - \sqrt{3r} +1),
  \]
  and $\Ree{r}$ contains subgroups isomorphic to $\PSL 2 {\FF_r}$,
  viz.\ subgroups of index $2$ in centralizers of involutions, see
  \cite[page\,62]{MR197587}.  If $p$ divides $r^2-1$, then
  $\PSL 2 {\FF_r}$ contains a Sylow $p$-subgroup of $\Ree{r}$, and
  every element of $\PSL 2 {\FF_r}$ is a product of two elements
  (transvections) of order $3$ by \cite[3.4]{MR3090721} or
  \cite[2.7]{MR0467511}.

It remains to consider primes $p>3$ that divide $r \pm \sqrt{3r} +1$;
this includes the prime divisor $7$ of $|\Ree 3|$. The corresponding
Sylow $p$-subgroups are cyclic, hence all subgroups of order $p$
are conjugate, and $\Ree{r}$ contains the
Frobenius group $C_p \rtimes C_3$ of order $3p$, see \cite[IV.3, page 83]{MR197587}. 
The inclusion $C_p \rtimes C_3 \le \AGL 1 {\FF_p}$ yields that
every element of order $p$ in $C_p \rtimes C_3$ is a
commutator, hence it is the product of two conjugate elements of order~$3$.
\end{proof}

\begin{theo}\label{ReeSemiReg}
  If\/~$\lpg$ is a Ree group then $\gp = \lpg$, and the action of\/~$\gp$
  on $U\setminus B_\infty$ is semi-regular. %
\end{theo}

\begin{proof} We have
  $\lpg =\Ree{r}$ with $r=3^{2e+1}\ge 3$, and the unital has order $q=r^3$.
  
  We first prove that $\gp = \lpg$ if $r=3$; then
  $\lpg = \Ree{3} \cong \PgL2{\FF_8}\cong \SL2{\FF_8} \rtimes C_3$.
  As in~\cite[Case\,(6), p.\,429]{MR3090721}, we note that the final
  term~$D$ of the commutator series of~$\gp$ is a cover of
  $\SL 2 {\FF_8}$, which has no proper cover
  (see~\cite[V.25.7]{MR0224703} or~\cite{MR0153677}), %
  hence $D \cong \SL 2 {\FF_8}$. %
  There exists a translation $\alpha \in \gp \setminus D$ of order $3$
  such that $\langle \alpha, D \rangle = \langle \alpha\rangle \ltimes D$
  induces $\lpg \cong G / \gp_{[B_\infty]}$ on $B_\infty$, as in
  \cite[Case\,(6), p.\,429]{MR3090721}.
  Hence $\gp$ is the direct product of $\langle \alpha\rangle \ltimes D$ with
  the center~$\gp_{[B_\infty]}$ of~$\gp$.
  Each Sylow $3$-subgroup of $\langle \alpha\rangle \ltimes D$ has order $3^3$
  and acts faithfully on $B_\infty$, hence it is a full translation group $\tg{z}$ for
  some $z\in B_\infty$. There exist at least two such Sylow $3$-subgroups
  (as $D$ is simple), and together they generate~$G$. Hence 
  $G = \langle \alpha\rangle \ltimes D$ and $\gp_{[B_\infty]}$ is trivial.
  
  For $r>3$, the Ree group $\Ree{r}$ is simple and has trivial Schur
  multiplier; see ~\cite{MR0193141}. So~$\gp = \lpg$ holds for
  every~$r\ge 3$, and the Sylow $3$-subgroups of $G$ are the translation groups.
  By \ref{productsRee} the stabilizer $G_c$ of a point $c \in U\setminus B_\infty$
  cannot contain any element of prime order, hence $G_c$ is trivial.
\end{proof}

\section{Proof of the Main Theorem}

Let\/ $\UU$ be a unital of order~$q$ with two points which are centers
of translation groups of order~$q$. %
Then the group~$\gp$ generated by these two translation groups induces
a Moufang set (as defined in Section~\ref{sec:unitals}) on the
block~$B_\infty$ containing the translation centers
(see~\cite[3.1]{MR3090721}, where our present group~$G$ is
called~$\hat{G}$). %
We have listed the possibilities for the little projective
group~$\lpg$ in~\ref{finiteMoufang}. %
The group~$\lpg$ cannot be a unitary group, see~\ref{excludeSU}. %
If~$\lpg$ is sharply two-transitive then $q\in\{2,3\}$ and~$\UU$ is
the hermitian unital of order~$q$, see~\ref{sharplyTwoTrs}; thus
$G\cong\SL2{\FF_q}$. %

Now assume that~$\lpg$ is isomorphic to $\PSL 2{\FF_q}$ but not
sharply two-transitive. Then $q>3$, the group~$\lpg$ is simple,
and~$\gp$ is a perfect central extension of~$\lpg$
by~\ref{PropSimple}. In most cases, the Schur multiplier
of~$\SL2{\FF_q}$ is trivial. In those cases, we have
$\gp \cong \SL2{\FF_q}$ or $\gp \cong \PSL2{\FF_q}$.  The
Schur multiplier of $\SL2{\FF_q}$ is not trivial only if
$q\in\{4,9\}$.  In these cases, the arguments in~\cite[p.\,428,
(2)]{MR3090721} show that $\gp\cong\SL2{\FF_4}$ if $q=4$ and
$\gp\cong\SL2{\FF_9}$ or $\gp\cong\PSL2{\FF_9}$ if $q=9$. %
By~\cite[3.5]{MR3090721}, the action of~$\gp$ on $U\setminus B_\infty$
is semi-regular (this also applies if $q\le3$).

If~$\lpg$ is a Suzuki or Ree group then $\gp=\lpg$, and the action on
$U\setminus B_\infty$ is semi-regular, see~\ref{SuzukiSemiReg}
and~\ref{ReeSemiReg}. The smallest Suzuki group
$\Sz{2} \cong \AGL1{\FF_5}$ is sharply two-transitive, and excluded
by~\ref{sharplyTwoTrs}. %

In each one of the cases discussed above, the order~$q$ of the unital
turns out to be a prime power (thanks to the restriction $q\in\{2,3\}$
in the sharply two-transitive case).

\section{Simplifications of a previous paper}

The present paper yields some simplifications of the classification 
of the unitals admitting all translations in \cite{MR3090721}, as we explain
now. The elimination of the sharply two-transitive groups in \ref{sharplyTwoTrs}
leaves only Moufang sets which are determined uniquely by the isomorphism
type of their root groups, see \cite[3.3]{MR3090721}. Thus the mapping
$g\colon \cL_c \to \mathbb N$ considered in \cite[page 430]{MR3090721}
is constant, and Proposition~4.2 in \cite{MR3090721} is not needed
anymore; %
the proof of that proposition depends on the classification of the finite
simple groups. By \ref{excludeSU} one can omit the consideration of
unitary groups.

The classification of the finite simple groups is still involved at
the very end of the proof in \cite[page 430]{MR3090721}, when we quote
a result of Kantor's which uses the classification of finite doubly
transitive groups.  If the order $q$ of the unital is a power of~$2$,
then the classification of finite simple groups can be avoided,
because the doubly transitive groups of degree $q^3+1$ are classified
in \cite[Theorem~2]{MR575516}; see also \cite{MR515456}.  %

\goodbreak


\providecommand{\noopsort}[1]{}\def\cprime{$'$}
  \def\polhk#1{\setbox0=\hbox{#1}{\ooalign{\hidewidth
  \lower1.5ex\hbox{`}\hidewidth\crcr\unhbox0}}}

\end{document}